\documentclass[a4paper,11pt]{amsart}
\usepackage{amssymb}
\usepackage{amsmath}
\usepackage{amsthm}
\usepackage{verbatim}
\usepackage{url}
\usepackage[all]{xy}

\theoremstyle{plain}

\newtheorem{theorem}{Theorem}[section]

\newtheorem{proposition}[theorem]{Proposition}
\newtheorem{lemma}[theorem]{Lemma}
\newtheorem{corollary}[theorem]{Corollary}

\newtheorem*{property*}{}

\theoremstyle{definition}

\newtheorem{subsec}[theorem]{}

\theoremstyle{remark}
\newtheorem*{notation*}{Notation}
\newtheorem{remark}[theorem]{Remark}

\newcommand{\RR}{{\mathbb{R}}}
\newcommand{\R}{\RR}

\newcommand{\CC}{{\mathbb{C}}}
\newcommand{\C}{\CC}

\newcommand{\QQ}{{\mathbb{Q}}}
\newcommand{\Q}{\QQ}

\newcommand{\ZZ}{{\mathbb{Z}}}
\newcommand{\Z}{\ZZ}

\newcommand{\GG}{{\mathbb{G}}}
\newcommand{\G}{\GG}

\newcommand{\kbar}{{\bar k}}

\newcommand{\Gbar}{{\overline{G}}}
\newcommand{\Qbar}{{\overline{\Q}}}

\newcommand{\GL}{{\rm GL}}

\newcommand{\gGer}{{\mathfrak{g}}}
\newcommand{\hh}{{\mathfrak{h}}}
\newcommand{\Lie}{{\rm Lie\,}}
\newcommand{\Ad}{{\rm Ad}}

\newcommand{\Hbar}{\overline{H}}

\newcommand{\sT}{{\mathcal{T}}}
\newcommand{\Gal}{{\mathrm{Gal}}}
\newcommand {\Res}{{\mathrm{Res}}}

\newcommand{\X}{{{\sf X}}}

\newcommand{\into}{\hookrightarrow}
\newcommand{\isoto}{\overset{\sim}{\to}}

\newcommand{\labelto}[1]{\xrightarrow{\makebox[1.5em]{\scriptsize ${#1}$}}}

\newcommand{\emm}{\bfseries}
\newcommand{\hs}{\kern 0.8pt}

\newcommand{\red}{{\rm red}}

\newcommand{\loc}{{\rm loc}}

\newcommand{\sN}{{\mathcal N}}

\newcommand{\lra}{{\,\longrightarrow\,\,}}

\newcommand{\sH}{{\mathcal{H}}}

\title[Conjugation of semisimple subgroups]
{Conjugation of semisimple subgroups\\ over real number fields of bounded degree}

\author{Mikhail Borovoi}
\address{Borovoi: Raymond and Beverly Sackler School of Mathematical Sciences,
Tel Aviv University, 6997801 Tel Aviv, Israel}
\email{borovoi@tauex.tau.ac.il}
\thanks{Borovoi was partially supported by the Hermann Min\-kow\-ski Center for Geometry
and by the Israel Science Foundation, grant No. 870/16.}

\author{Christopher Daw}
\address{Daw: Department of Mathematics and Statistics, University of Reading,
    White\-knights,  PO Box 217,  Reading,  Berkshire RG6 6AH,  United Kingdom}
\email{chris.daw@reading.ac.uk}
\thanks{Daw is grateful to the University of Reading for financial support}

\author{Jinbo Ren}
\address{ Ren:
  Institut des Hautes \'Etudes Scientifiques, Le Bois-Marie 35, route de Chartres, 91440 Bures-sur-Yvette, France}
\curraddr{Department of Mathematics, University of Virginia, Charlottesville, VA 22904-4137, USA}
\email{jr9zb@virginia.edu}
\thanks{Ren was partially supported by grants from R\'egion \^Ile de France}



\begin{document}

\begin{abstract}
Let $G$ be a linear algebraic group over a field $k$ of characteristic 0.
We show that any two connected semisimple $k$-subgroups of $G$
that are conjugate over an algebraic closure of $k$
are actually conjugate over a finite field extension of $k$
of degree bounded independently of the subgroups.
Moreover, if $k$ is a {\em real number field},
we show that any two connected semisimple $k$-subgroups of $G$
that are conjugate over the field of real numbers $\R$ are actually conjugate
over a finite {\em real} extension of $k$ of degree
bounded independently of the subgroups.
\end{abstract}

\maketitle
\tableofcontents

\setcounter{section}{-1}

\section{Introduction}

In this article, by a semisimple algebraic group
we always mean a {\em connected} semisimple group.

Let $k$ be a field and let $G$ be a linear algebraic group over $k$.
Let $H$ and $H'$ be two algebraic $k$-subgroups of $G$.
Let $K/k$ be a field extension.
We write $H_K=H\times_k K$ and $H'_K=H'\times_k K$.
We say that $H$ and $H'$ are {\em conjugate over  $K$}
if there exists an element $g\in G(K)$ such that
\begin{equation*}
H'_K =g\cdot H_K \cdot g^{-1}.
\end{equation*}

We start with the following finiteness results for the set
of conjugacy classes of semisimple subgroups:

\begin{proposition} \label{finiteconjugation}
Let $G$ be a linear algebraic group over an algebraically closed field  $k$ of characteristic 0.
Then the set $C(G)$ of $G(k)$-conjugacy classes of semisimple $k$-subgroups of $G$ is finite.
\end{proposition}

\begin{corollary}\label{finiteconjugation-R}
Let $k$ be a field of characteristic 0 of type (F)
in the sense of Serre \cite{SerreCG}, III.4.2, for example,
the field of real numbers $\R$ or a $p$-adic field
(a finite extension of the field of $p$-adic numbers $\Q_p$).
Let $G$ be a linear algebraic group over $k$.
Then the set $C(G)$ of $G(k)$-conjugacy classes of semisimple $k$-subgroups of $G$ is finite.
\end{corollary}

Proposition \ref{finiteconjugation} and Corollary \ref{finiteconjugation-R}
will be proved in Section \ref{s:conjugation}.

Our main results are Theorems \ref{t:abs} and \ref{t:abs-R} below.

\begin{theorem}\label{t:abs}
Let $k$ be a field of characteristic 0 and let $\kbar$ be a fixed algebraic closure of $k$.
Let $G$ be a linear algebraic group over $k$.
There exists a natural number $d$ depending only on $G_\kbar$ with the following property:
if $H$ and $H'$ are two semisimple $k$-subgroups of $G$ that are conjugate over $\kbar$,
then they are conjugate over a finite extension $K/k$ of degree $[K:k]\le d$.
\end{theorem}

In Section \ref{s:reduction}, we shall deduce Theorem \ref{t:abs}
from Proposition \ref{finiteconjugation}  and  the following theorem:

\begin{theorem}\label{t:cohom}
Let $k$ be a field of characteristic 0 and let $\kbar$ be a fixed algebraic closure of $k$.
Let $N$ be a linear algebraic group over  $k$.
Then there exists a natural number $d=d(N_\kbar)$ such that
any cohomology class $\xi\in H^1(k,N)$ can be killed
by a finite field extension of degree at most $d$
(that is, there exists a field extension $K=K(N,\xi)$ of $k$ of degree $[K:k]\le d$
such that the image of $\xi$ in $H^1(K,N)$ is 1).
\end{theorem}

We shall prove Theorem \ref{t:abs} in Section \ref{s:reduction}
by applying Theorem \ref{t:cohom} to the normalizer $N=\mathcal{N}_G(H)$ of $H$ in $G$.
For a proof of Theorem  \ref{t:cohom}, see Section \ref{s:cohom}.

We need a version of Theorem \ref{t:abs} for real number fields.
A {\em real number field} is a finite extension $k$ of $\Q$
contained in the field of real numbers $\R$, or, in other words,
a finite extension $k$ of $\Q$ equipped with an embedding into $\R$.

\begin{theorem}\label{t:abs-R}
Let $G$ be a linear algebraic group over a {\em real number field} $k\subset \R$.
There exists a natural number $d_\R$ depending only on $G_\C$ with the following property:
if $H$ and $H'$ are two semisimple $k$-subgroups of $G$ that are conjugate over $\R$,
then they are conjugate over some  finite {\em real extension}
$K/k$, \ $k\subset K\subset \R$,  of degree $[K:k]\le d_\R$.
\end{theorem}

In Section \ref{s:reduction}, we shall deduce Theorem \ref{t:abs-R}
from Corollary \ref{finiteconjugation-R}  and  the following theorem:

\begin{theorem}\label{t:cohom-R}
Let $k\subset \R$ be a real number field.
Let $N$ be a linear algebraic group over  $k$.
Then there exists a natural number $d_\R=d_\R(N_\C)$ such that
any cohomology class $\xi\in H^1(k,N)$ becoming 1 over $\R$
can be killed by a finite real extension of degree at most $d_\R$.
\end{theorem}

We shall prove Theorem \ref{t:cohom-R} in Section \ref{s:cohom-R}.
This is the most complicated proof  of the present article.
The current proof, proposed by an anonymous referee,
uses a result of G.~Lucchini Arteche \cite{LA15} (see Proposition \ref{p:LA15})
and real approximation for homogeneous spaces with cyclic finite stabilizers
(see Theorem \ref{p:LA-cyclic}).
Theorem \ref{t:cohom-R} can be  also proved without
using the cited result of Lucchini-Arteche (see version 1 of \cite{BDR}).

\begin{subsec}
Now let $k$ be a real number field and let $G$ be a linear algebraic group over  $k$.
Consider the set $C(G_\R)$ of $G(\R)$-conjugacy classes of semisimple $\R$-subgroups of $G$;
by Corollary \ref{finiteconjugation-R}, this set is finite.
Let $C_0\subset C(G_\R)$ denote the set of those $c\in C(G_\R)$
for which there exists a semisimple subgroup in $c$ {\em defined over $k$}.
For each $c\in C_0$\hs, let us choose such a semisimple $k$-subgroup $H_c\subset G$ in $c$.
We obtain a finite set of subgroups $\Omega=\{H_c\ |\ c\in C_0\}$ with the following property:
any semisimple $k$-subgroup of $G$ is conjugate over $\R$ to some $H_c\in\Omega$.
The next corollary follows immediately from Theorem \ref{t:abs-R}.
\end{subsec}

\begin{corollary}\label{finite-real}
Let $k$ be a real number field and let $G$ be a linear algebraic group over  $k$.
Let $\Omega=\{H_c\ |\ c\in C_0\}$ be a finite set of semisimple $k$-subgroups of $G$ as above.
Then there exists a natural number $d=d(G_\C)$
such that any semisimple $k$-subgroup  $H\subset G$ is conjugate to some $H_c$ ($c\in C_0$)
over a finite real extension $K/k$, $k\subset K\subset\R$, of degree $[K:k]\le d$.
\end{corollary}

\noindent
{\em Motivation.}
This article was motivated by earlier work of Daw and Ren
on the Zilber--Pink conjecture for Shimura varieties.
The relevance of Corollary \ref{finite-real} is explained
in Section 12 of their article \cite{DawRen17}.

\begin{notation*}
Let $k$ be a field of characteristic 0.
In this article, by a $k$-variety we mean a separated scheme
of finite type over $k$, not necessarily irreducible.
By an algebraic group over $k$, or, shorter, a $k$-group,
we always mean a {\em linear} algebraic group over $k$,
that is, an affine group scheme of finite type over $k$, not necessarily connected.
If $G$ is a $k$-group, we write $G^0$ for the identity component of $G$.
If, moreover, $G$ is connected, we denote by $R_u(G)$ the unipotent radical of $G$.
By a $k$-subgroup of $G$ we mean an algebraic $k$-subgroup of $G$.
\end{notation*}

\paragraph{\em Acknowledgements.}
Jinbo Ren would like to thank his supervisor Emmanuel Ullmo for regular discussions
and constant support during the preparation of this article
and he would like to thank Yongqi Liang for several useful discussions.
The authors are grateful to  Friedrich Knop for his
MathOverflow answer \cite{Knop-MO} to Jinbo Ren's question,
to Sean Lawton for the reference to Richardson's article \cite{Richardson},
and to Giancarlo Lucchini Arteche for useful comments.
We thank the anonymous referees for their helpful comments.
We especially thank one of the  referees for his/her suggestion
to use a result of Lucchini Arteche
and a result on real approximation for homogeneous spaces with cyclic finite stabilizers,
which permitted us to shorten the proofs significantly.

\section{Conjugation over an algebraically closed field\\ and over a  field of type (F)}
\label{s:conjugation}

\begin{subsec}
{\em Proof of Proposition \ref{finiteconjugation}.}
Write $\gGer=\Lie G$; the group $G(k)$ acts on $\gGer$ via the adjoint representation.
Let $\hh\subset \gGer$ be a Lie subalgebra; $\hh$ acts on itself,
on $\gGer$, and on the quotient space $\gGer/\hh$,
so we may consider the cohomology space $H^1(\hh,\gGer/\hh)$.
By Richardson \cite{Richardson}, Corollary (b) in the introduction,
there exist only finitely many $G(k)$-conjugacy classes of Lie subalgebras $\hh\subset \gGer$
such that $H^1(\hh,\gGer/\hh)=0$.

If $\hh$ is semisimple, then for any $\hh$-module $M$  we have $H^1(\hh,M)=0$
(see Chevalley and Eilenberg \cite{Chevalley-Eilenberg}, Theorem 25.1).
In particular, we have $H^1(\hh,\gGer/\hh)=0$.
It follows that there exist only finitely many $G(k)$-conjugacy classes
of {\em semisimple} Lie subalgebras of $\gGer$.

Now let $H_1$ and $H_2$ be two semisimple subgroups of $G$.
Let $\hh_1$ and $\hh_2$ denote their respective Lie algebras.
Assume that $\hh_1$ and $\hh_2$ are conjugate under $G(k)$, that is, there exists $g\in G(k)$
such that $\Ad(g)(\hh_1)=\hh_2$.
Since ${\rm char}(k)=0$, a connected algebraic subgroup $H\subset G$
is uniquely determined by its Lie algebra $\Lie H\subset \gGer$ (see  Humphreys \cite{Humphreys}, Section 13.1).
It follows that $g\hs H_1\hs g^{-1}= H_2$.
We see that if $\Lie H_1$ and $\Lie H_2$ are conjugate, then $H_1$ and $H_2$ are conjugate.
We conclude that there are only finitely many conjugacy classes of semisimple subgroups of $G$,
which proves the proposition.
\qed
\end{subsec}

\begin{remark}
The properties `connected' and `semisimple' required of the subgroups in Proposition \ref{finiteconjugation} are necessary.
Indeed, let $G$ be  the two-dimensional torus $\G_{m,k}^2$\hs,
where $\G_{m,k}$ denotes the multiplicative group over $k$.
Then clearly $G$ has infinitely many different $k$-subgroups,
and even infinitely many different connected $k$-subgroups,
and, clearly,  they are not conjugate in $G$.
Moreover, let $n\in\Z_{>0}$. Set
\[S_n=\{(x_1,x_2)\in \G_{m,k}^2\ |\ x_2=x_1^n\}\hs,\quad H_n=\G_{a,k}^2\rtimes S_n\subset \G_{a,k}^2\rtimes \G_{m,k}^2\hs,\]
where $\G_{a,k}$ denotes the additive group over $k$, on which $\G_{m,k}$ naturally acts.
We embed
\[H_n\subset \G_{a,k}^2\rtimes \G_{m,k}^2\into \GL_{4,k}\into {\rm SL}_{5,k}\hs.\]
One can easily check that the connected solvable algebraic groups $H_n$ are pairwise non-isomorphic, and therefore,
they are not conjugate in the simple group ${\rm SL}_{5,k}$.
\end{remark}

\begin{subsec}\label{ss:gen-1}
Let $k$ be a field of characteristic 0 and let $\kbar$ be a fixed algebraic closure of $k$.
Let $G$ be a linear algebraic group over $k$.
Let $C$ be a conjugacy class of connected $\kbar$-subgroups $\Hbar'$ of $G_\kbar$
that contains a subgroup  $H_\kbar=H\times_k \kbar$ defined over $k$.
Then $C$ is the set of $\kbar$-points of a $k$-variety $V$ on which $G$ acts transitively  by
\[ g*\Hbar'=g\cdot \Hbar'\cdot g^{-1}.\]
The stabilizer of the $k$-point $H$ of this variety is $N=\mathcal{N}_G(H)$.
Therefore, we may identify $V$ with $G/N$.
We identify the set of $k$-subgroups $H'\subset G$ that are conjugate to $H$ over $\kbar$ with $(G/N)(k)$,
and we identify the set of $G(k)$-conjugacy classes of such $k$-subgroups
with the set of orbits of $G(k)$ in $(G/N)(k)$.
\end{subsec}

\begin{subsec}
{\em Proof of Corollary \ref{finiteconjugation-R}.}
Let $k$ be a field of characteristic 0 of type (F) in the sense of Serre.
By Serre \cite{SerreCG}, III.4.4, Theorem 5, the set of orbits of $G(k)$
in $(G/N)(k)$ in Subsection \ref{ss:gen-1} is finite.
Therefore, the set of $G(k)$-conjugacy classes of connected $k$-subgroups in a $G(\kbar)$-conjugacy class is finite.

Let $C_k$ denote the set of $G(k)$-conjugacy classes of connected algebraic $k$-sub\-groups,
and let $C_\kbar$ denote the corresponding set for $\kbar$. Then we have a canonical map
\[ C_k\to C_\kbar.\]
As explained above, since $k$ is of type (F), all fibers of this map are finite.
By Proposition \ref{finiteconjugation}, the set of $G(\kbar)$-conjugacy classes
of connected {\em semisimple} $\kbar$-subgroups is finite,
and the corollary follows.
\qed
\end{subsec}

\section{Reductions}
\label{s:reduction}

\begin{subsec}
We show that, in order to prove Theorem \ref{t:abs}, it suffices to prove Theorem \ref{t:cohom}.
Indeed, by Proposition \ref{finiteconjugation}, the set $C$ of conjugacy classes
of semisimple subgroups of $G_\kbar$ is finite.
Therefore, it suffices to show that, if $H,H'\subset G$ are two such $k$-subgroups
{\em in a given $G(\kbar)$-conjugacy class $c\in C$},
then they are conjugate over a finite extension $K$ of $k$ of degree at most $d_c$\hs,
where $d_c$ depends only on $G_\kbar$ and $c$.

Set
\[Y=\{\hs g\in G(\kbar)\ |\ gHg^{-1}=H'\hs\}\]
and write $N=\sN_G(H)$.
Since $H$ and $H'$ are connected and ${\rm char}(k)=0$,
by Humphreys \cite{Humphreys}, Theorem 13.1, we have
\[Y=\{\hs g\in G(\kbar)\ |\ \Ad(g) (\Lie H)=\Lie H'\hs\},\]
and hence, the variety $Y$ is defined over $k$.
The $k$-group $N$ acts on the right on $Y$ by
\[g\mapsto gn \qquad (g\in G(\kbar),\ n\in N(\kbar)).\]
This action is simply transitive (over $\kbar$).
Hence, $Y$ is a principal homogeneous space (torsor) of $N$,
and so we obtain a cohomology class $\xi=\xi(H,H')\in H^1(k,N)$.
The two subgroups $H$ and $H'$ are conjugate over a field extension $K$ of $k$
if and only if $Y$ has a $K$-point,
hence, if and only if the extension $K/k$ kills $\xi$.
Since, by Theorem \ref{t:cohom}, the class $\xi$ can be killed
by a finite field extension $K/k$ of degree at most $d(N_\kbar)$,
the subgroups $H$ and $H'$ are conjugate over this field $K$.
Clearly, $N_\kbar$ depends (up to isomorphism) only on $G_\kbar$ and $c$.
\qed
\end{subsec}

\begin{subsec}
We show that, in order to prove Theorem \ref{t:abs-R},
it suffices to prove Theorem \ref{t:cohom-R}.
Indeed, by Corollary \ref{finiteconjugation-R}, the set $C_\R$ of $G(\R)$-conjugacy classes
of semisimple subgroups of $G_\R$ is finite.
Therefore, it suffices to show that, if $H,H'\subset G$ are two semisimple $k$-subgroups of $G$
{\em in a given $G(\R)$-conjugacy class $c\in C_\R$}, then they are conjugate
over  a finite real  extension $K$ of $k$ of degree at most $ d_{\R,c}$,
where $d_{\R,c}$ depends only on $G_\C$ and $c$.

As above, starting from two semisimple $k$-subgroups $H,H'\subset G$ that are conjugate over $\R$,
we obtain a cohomology class
\[\xi\in\ker[H^1(k,N)\to H^1(\R,N)],\]
where $N=\mathcal{N}_G(H)$.
By Theorem \ref{t:cohom-R}, the class $\xi$ can be killed
by a finite real extension $K/k$ of degree at most $d_\R(N_\C)$.
Then the subgroups $H$ and $H'$ are conjugate over this field $K$.
Clearly, $N_\C$ depends (up to isomorphism) only
on $G_\C$ and the conjugacy class of $H$ over $\C$.
\qed
\end{subsec}

\section{Taking quotient by a unipotent normal subgroup}
\label{s:Serre-Sansuc}

In this section $k$ is a field of characteristic 0.
We shall need the following result of Sansuc:

\begin{proposition}[\cite{Sansuc}, Lemma 1.13]
\label{l:Sansuc}
Let $G$ be an algebraic group over a  field $k$ of characteristic 0,
and let $U\subset G$ be a unipotent $k$-subgroup.
Assume that $U$ is normal in $G$.
Then the canonical map $H^1(k,G)\to H^1(k,G/U)$ is bijective.
\end{proposition}

Sansuc assumes that $G$ is connected, but his proof does not use this assumption.
As Sansuc's proof is concise,  we provide a more detailed  proof.

We need a lemma.

\begin{lemma}[well-known]
\label{l:comm-unip}
Let $U$ be a {\emm commutative} unipotent algebraic group over a field $k$ of characteristic 0. Then:
\begin{enumerate}
\item[(i)] $U\simeq \G_{a,k}^d$, where  $\G_{a,k}$ is the additive group
and $d\ge 0$ is an integer (the dimension of $U$).
\item[(ii)] Any twisted form of $U$ is isomorphic to $U$.
\item[(iii)] $H^n(k,U)=0$ for all $n\ge 1$.
\end{enumerate}
\end{lemma}

\begin{proof}
For (i)  see, for instance,  Milne \cite[Corollary 14.33]{Milne}.
Any twisted form of $U=\G_{a,k}^d$ is again a commutative unipotent $k$-group
of the same dimension $d$, and by (i)
it is isomorphic to $\G_{a,k}^d$, whence (ii).
By Serre \cite{SerreLF}, X.1, Proposition 1, we have $H^n(k,\G_{a,k})=0$ for all $n\ge1$, and (iii) follows.
\end{proof}

\begin{subsec}
{\em Proof of Proposition \ref{l:Sansuc}.}
We prove the proposition by induction on the dimension of $U$.
We may and shall assume that $U\ne \{1\}$.
Let $Z=Z(U)$ denote the center of $U$. Then $Z\neq 1$ because $U$ is nilpotent.
Clearly, $Z$ is a commutative unipotent group.
By Lemma  \ref{l:comm-unip}(i) we have $\dim Z>0$.
Since $Z$ is a characteristic subgroup of $U$, we see that $Z$ is normal in $G$.
Set $U'=U/Z$ and $G'=G/Z$.
Then $U'$ is a normal unipotent subgroup of $G'$ and $\dim \,U'<\dim\, U$.
We factorize the canonical homomorphism $G\to G/U$ as
\[ G\labelto{p} G'\labelto{q} G'/U'=G/U.\]
We obtain a  factorization of the canonical map $H^1(k,G)\to H^1(k,G/U)$ as
   \[H^1(k,G) \labelto{p_*} H^1(k,G') \labelto{q_*} H^1(k,G'/U')=H^1(k,G/U).\]
Since $\dim U'<\dim U$, by the induction hypothesis the map $q_*$ is bijective.
It remains to show that the map $p_*$ is bijective.

Consider the short exact sequence of $k$-groups
   \[1\lra  Z\labelto{i}  G\labelto{p}  G/Z\lra 1\]
and the induced cohomology exact sequence
\[H^1(k,Z)\labelto{i_*} H^1(k,G)\labelto{p_*} H^1(k,G/Z).\]
Since by Lemma \ref{l:comm-unip} we have  $H^1(k,\hs_c Z)=0$
and $H^2(k,\hs_c Z)=0$ for any cocycle $c\in Z^1(k,G/Z)$,
by Serre \cite{SerreCG}, I.5.5, Corollary 2 of Proposition 39, and
I.5.6, Corollary of Proposition 41, the map  $p_*$ is bijective,
which completes the proof of the proposition. \qed
\end{subsec}

\section{Killing a first cohomology class over an arbitrary field\\ of characteristic 0}
\label{s:cohom}

In this section we prove Theorem \ref{t:cohom}, which we state here in a more precise form.

\begin{theorem}\label{t:cohom-bis}
Let $\kbar$ be an algebraically closed field of characteristic 0.
Let $\Gbar$ be a linear algebraic group over  $\kbar$.
Then there exists a natural number $d=d(\Gbar)$ with the following property:
\begin{property*}[$*$]
For any subfield $k\subset \kbar$ such that $\kbar$
is an algebraic closure of $k$, for any $k$-form $G$ of $\Gbar$,
and for  any cohomology class $\xi\in H^1(k,G)$, the class $\xi$ can be killed
by a finite field extension $K/k$ of degree at most $d$.
\end{property*}
\end{theorem}

First, we prove Theorem \ref{t:cohom-bis}  for finite groups.

\begin{lemma}\label{l:finite-k}
Let $\Gbar$ be a finite $\kbar$-group. Then $\Gbar$ satisfies $(*)$ with $d=|\Gbar|$.
\end{lemma}

\begin{proof}
Let $G$ be a $k$-form of $\Gbar$. A cohomology class $\xi\in H^1(k,G)$ defines a torsor $\sT=\sT_\xi$ of $G$.
By definition, we have $|\sT(\kbar)|=|\Gbar|$.
For any point $t\in \sT(\kbar)$, the $\Gal(\kbar/k)$-orbit of $t$ has cardinality at most $|\Gbar|$,
hence, $t$ is defined over a finite extension of $k$ in $\kbar$ of degree at most $|G|$.
Since this extension kills $\xi$, the proof is complete.
\end{proof}

In what follows, we shall need two known results.

\begin{proposition}[Minkowski, cf.~Friedland \cite{finitegln}]
\label{p:minkowski}
For any natural number $r$, there is a constant $\beta=\beta(r)>0$
such that every finite subgroup of $\GL_r(\ZZ)$ has cardinality at most $\beta$.
\end{proposition}

\begin{proposition}[Lucchini Arteche \cite{LA15}, Corollary 18]
\label{p:LA15}
Let $k$ be a field of characteristic 0.
Let $G$ be a linear algebraic $k$-group with reductive identity component.
Let $T$ be a maximal $k$-torus in the identity component $G^0$ of $G$, and let $W=\sN_G(T)/T$,
which is a finite group.
We write $r$ for the dimension of $T$, and write $w$ for the order of $W$.
Let $L/k$ be a field extension splitting $T$ whose degree we denote by $d_T$.
Then there exists a finite $k$-subgroup $H$ of $G$ of order at most $w^{2r+1}d_T^{2r}$ with the property
\begin{property*}[\rm SUR]
For every field  $K\supset k$,  the natural map $H^1(K,H)\to H^1(K,G)$ is surjective.
\end{property*}
\end{proposition}

\begin{corollary}\label{c:LA}
Let $\kbar$ be an algebraically closed field of characteristic 0,
and let $\Gbar$ be a linear algebraic group  over $k$ with reductive identity component.
Let $k\subset \kbar$ be a subfield.
Then there exists a positive integer $\gamma(\Gbar)$ depending only on $\Gbar$
such that, for any $k$-form $G$ of $\Gbar$ there exists
a finite $k$-subgroup $H$ of $G$ of order at most  $\gamma(\Gbar)$ with
the property {\rm (SUR)} of Proposition \ref{p:LA15}.
\end{corollary}

\begin{proof}
In Proposition \ref{p:LA15} we  may take for $L$ the minimal splitting field of the maximal $k$-torus $T$.
Then the Galois group $\Gal(L/k)$ acts faithfully on the character group $\X^*(\overline{T})\simeq \Z^r$,
and hence,
\[d_T=[L:k]=\#\Gal(L/k)\le\beta(r),\]
cf.~Proposition \ref{p:minkowski}.
Clearly, $r$ and $w$ in Proposition \ref{p:LA15} depend only on $\Gbar$,
and we may take $\gamma(\Gbar)=w^{2r+1}\hs \beta(r)^{2r}$.
\end{proof}

\begin{subsec} \label{ss:pf-gen-0}
\emph{Reduction of Theorem \ref{t:cohom-bis} to the case when $\Gbar^0$ is reductive.}
Let $G$ be a $k$-form of $\Gbar$.
Consider the normal unipotent subgroup $R_u(G^0)\subset G$,
where $^0$ denotes the identity component and $R_u$ denotes the unipotent radical.
Write $G^\red=G/R_u(G^0)$.
Similarly, write $\Gbar^\red=\Gbar/R_u(\Gbar^0)$.
It is clear that  the identity component of $\Gbar^\red$ is reductive and that
 $G^\red$ is a $k$-form of $\Gbar^\red$.
Let $\pi\colon  G\to G^\red$ denote the canonical surjective homomorphism.
By Proposition \ref{l:Sansuc}, for any field extension $K/k$,
the canonical map
$\pi_*\colon H^1(K,G)\to H^1(K,G^\red)$ is bijective.
We see that for any $\xi\in H^1(k,G)$, an extension $K/k$ kills $\xi$ if and only if it kills $\pi_*(\xi)$.
We conclude that if Theorem  \ref{t:cohom-bis} is true for  $\Gbar^\red$  with bound $d(\Gbar^\red)$,
then this theorem is also true for  $\Gbar$ with the same bound $d(\Gbar)=d(\Gbar^\red)$.
\end{subsec}

\begin{subsec}\label{ss:pf-gen} \emph{Proof of Theorem \ref{t:cohom-bis}.}
By \ref{ss:pf-gen-0} we may assume that $\Gbar^0$ is reductive.
Let $H\subset G$ be as in Proposition \ref{p:LA15}  (due to Lucchini Arteche).
By Corollary \ref{c:LA}  we may assume that the order of $\Hbar$ is at most  $\gamma(\Gbar)$.
Let $\xi\in H^1(k,G)$.
From the property (SUR) of $H$,  we know  that $\xi$ is the image of some cohomology class $\xi_H\in H^1(k,H)$.
By Lemma \ref{l:finite-k}, the class $\xi_H$, and hence $\xi$,
can be killed by a finite field extension of degree at most $d(\Hbar)=|\Hbar|\le\gamma(\Gbar)$.
This completes the proof of Theorem  \ref{t:cohom-bis}  with the  bound $d(\Gbar)=\gamma(\Gbar^\red)$,
and thus completes the proofs of Theorems  \ref{t:cohom} and \ref{t:abs}.
\qed
\end{subsec}

\section{Real approximation for homogeneous spaces}
\label{s:real}

In this section we prove the following theorem, which was communicated to us by an anonymous referee.

\begin{theorem}\label{p:LA-cyclic}
Let $G$ be a connected  linear algebraic group over a number field $k$,
and consider the homogeneous space $Y=G/C$,
where $C\subset G$ is  a cyclic finite   $k$-subgroup.
Then $Y$ has the real approximation property, that is,
for any  set $S$ of {\emm archimedean} places of $k$,
the set $Y(k)$ is dense in $\prod_{v\in S} Y(k_v)$.
In particular, if $k\subset \R$ is a real number field, then  $Y(k)$ is dense in $Y(\R)$.
\end{theorem}

Let  $k$ be a field of characteristic 0.
We refer to Colliot-Th\'el\`ene \cite{CT}, Definition 2.1 and Proposition 2.2, for the definition
of a {\em quasi-trivial $k$-group.}

\begin{lemma}[Colliot-Th\'el\`ene]
\label{l:CT}
Let $G$ be a connected  linear algebraic group over a field $k$ of characteristic 0.
Then there exists a short exact sequence of $k$-groups
     \[1\lra Z'\lra G'\labelto{\nu} G\lra 1\]
where $G'$ is a quasi-trivial  $k$-group and $Z'$ is
a $k$-group of multiplicative type contained in the center of $G'$.
\end{lemma}

\begin{proof}
See  Colliot-Th\'el\`ene \cite{CT}, Proposition-Definition 3.1, or \cite{Borovoi}, Proposition 2.8.
\end{proof}

\begin{proposition}\label{p:qt-ab}
Let $G$ be a connected linear algebraic group over a field $k$ of characteristic 0,
and let $C\subset G$ be a cyclic finite $k$-subgroup.
Then there exists an isomorphism of $k$-varieties $G'/H'\isoto G/C$,
where $G'$ is a quasi-trivial  $k$-group and $H'\subset G'$ is an abelian $k$-subgroup.
\end{proposition}

\begin{proof}
Consider the short exact sequence of Lemma \ref{l:CT},
and set $H'=\nu^{-1}(C)\subset G'$.  Clearly, we have $G/C=G'/H'$.
The $k$-subgroup $H'$ fits into a short exact sequence
\[ 1\to Z'\to H'\to C\to 1,\]
where $C$ is cyclic and $Z'$ is central in $H'$.
By Lemma \ref{l:sH} below, the group $H'$ is abelian.
\end{proof}

\begin{lemma}[well-known and very easy]
\label{l:sH}
Let an abstract group $\sH$ be a central extension of a cyclic group, that is,
we assume that $\sH$ fits into a short exact sequence
\[ 1\to A\to\sH\to C\to 1,\]
where $C$ is a cyclic group and $A$ is contained in the center of \hs$\sH$.
Then $\sH$ is abelian. \qed
\end{lemma}

\begin{subsec} {\em Proof of Theorem \ref{p:LA-cyclic}.}
By Proposition \ref{p:qt-ab} we have $Y=G'/H'$,
where $G'$ is a quasi-trivial $k$-group and $H'\subset G'$ is an abelian $k$-subgroup.
By \cite{Borovoi}, Corollary 2.3, any $k$-variety of the form $G'/H'$,
where $G'$ is a quasi-trivial $k$-group and $H'\subset G'$ is an abelian $k$-subgroup,
has the real approximation property.
\qed
\end{subsec}

\section{Killing a first cohomology class over a real number field}
\label{s:cohom-R}

In this section we prove Theorem \ref{t:cohom-R}, which we state here in a more precise form.

\begin{theorem}\label{t:cohom-R-bis}
Let $G_\C$ be a linear algebraic group over the field of complex numbers $\C$.
Then there exists a natural number $d_\R=d_\R(G_\C)$ with the following  property:
\begin{property*}[$*_\R$]
For any real number field $k\subset \R$, any $k$-form $G$ of $G_\C$, and
any cohomology class $\xi\in H^1(k,G)$ becoming 1 over $\R$,
the class $\xi$ can be killed by a finite real extension $K/k$ of degree at most $d_\R$.
\end{property*}
\end{theorem}

In order to prove Theorem \ref{t:cohom-R-bis},
we shall need the following elementary lemma:

\begin{lemma}\label{l:degree-two}
Let $k\subset \R$ be a real number field and let $L\subset \C$
be a finite {\emm normal} field extension of $k$.
Set $K=L\cap \R$. Then $[L:K]\le 2$.
\end{lemma}

\begin{proof}
If $L=K$, there is nothing to prove, so we assume that  $L\neq K$.
Write $\sigma\colon \C\to\C$ for the complex conjugation. Then $\sigma(k)=k$, because $k\subset\R$.
Since $L$ is a normal extension of $k$, we see that $\sigma(L)=L$.
Let $x\in L\smallsetminus K$, and set $p=(x+\sigma(x))/2$, $q=x\cdot \sigma(x)$.
Then $p,q\in K$ and $x^2-2px+q=0$.
Set $D=p^2-q\in K\subset \R$.
 Then $x=p+\zeta$, where $\zeta^2=D$.
Since
$x\in  L\smallsetminus K\ \subset\  \C\smallsetminus \R$, we have $D<0$.
Clearly,  $K(x)=K(\zeta)$.
If $x'\in  L\smallsetminus K$ is another element,
then similarly $x'\in K(\zeta')$, where $(\zeta')^2=D'<0$.
Then $(\zeta'/\zeta)^2=D'/D>0$, hence $\zeta'/\zeta\in L\cap \R=K$, and therefore, $x'\in K(\zeta)$.
Thus $L=K(\zeta)$ and $[L:K]=2$.
\end{proof}

We shall need a lemma and a corollary:

\begin{lemma}\label{finitebdd}
Let $X$ be a {\emm finite} variety over a real number field $k\subset \R$.
Let $x\in X(\R)$ be an $\R$-point.
Then $x$ is defined over  a real number field $k_x\subset \R$
whose degree over $k$ is at most $|X_\C|$,
where we write $|X_\C|$ for $|X(\C)|$.
\end{lemma}

\begin{proof}
Since $X$ is finite, we have $X(\C)=X(\kbar)$,
where $\kbar=\Qbar$ denotes the algebraic closure of $k$ in $\C$.
The Galois group  $\Gal(\kbar/k)$ acts on $X(\kbar)$.
The point $x$ is defined over the fixed field $k_x\subset \R$
of the stabilizer of $x$ in $\Gal(\kbar/k)$.
Since the cardinality of the orbit of $x$ under $\Gal(\kbar/k)$ is at most $|X(\kbar)|=|X_\C|$,
we see that  $[k_x:k]\leq |X_\C|$.
\end{proof}

\begin{corollary}\label{finite-real1finite}
Let $G_\C$ be a finite group over $\C$. Then $G_\C$ satisfies $(*_\R)$ with $d_\R=|G_\C|$.
\end{corollary}

\begin{proof}
Let $k\subset\R$ be a real number field, let $G$ be a $k$-form of $G_\C$,
and let $\xi\in H^1(k,G)$ be a cohomology class becoming 1 over $\R$.
Then $\xi$ defines a $k$-torsor $\sT=\sT_\xi$ of $G$ such that $\sT(\R)$ is nonempty.
Let $t\in \sT(\R)$.
By Lemma \ref{finitebdd} the torsor $\sT$ has a $k_t$-point
over a real number field $k_t\subset \R$
whose degree over $k$ is at most $|\sT_\C|=|G_\C|$.
The extension $k_t/k$ kills $\xi$, as required.
\end{proof}

Now we reduce Theorem \ref{t:cohom-R-bis} to the case when $G_\C$ is connected.

\begin{lemma}\label{assume-conn}
Assume that Theorem \ref{t:cohom-R-bis} is true for connected  linear algebraic groups,
and let $G_\C$ be any linear algebraic group over $\C$ (not necessarily connected).
Then $G_\C$ satisfies $(*_\R)$ with
\begin{align*}
d_\R=d_\R(G^0_\C)\cdot |G_\C/G^0_\C|^2.
\end{align*}
\end{lemma}

\begin{proof}
Let $k\subset\RR$ be a real number field, let $G$ be a $k$-form of $G_\C$,
and let $\xi\in H^1(k,G)$ be a cohomology class becoming 1 over $\R$.
From the short exact sequence
     \[1\longrightarrow G^0\longrightarrow G\longrightarrow G\slash G^0\longrightarrow 1\]
 we obtain a commutative diagram (of pointed sets) with exact rows:
\[
\xymatrix{
(G/G^0)(k) \ar[r]^\delta \ar[d]        &H^1(k,G^0)\ar[r]^\varphi \ar[d]^{\loc_\R}
       &H^1(k,G)\ar[r]^\psi\ar[d]^{\loc_\R} &H^1(k,G/G^0)\ar[d]^{\loc_\R}  \\
(G/G^0)(\RR) \ar[r]^{\delta_\R}                   &H^1(\RR,G^0)\ar[r]^{\varphi_\R}
       &H^1(\RR,G)\ar[r]^{\psi_\R}   &H^1(\R,G/G^0).
}
\]

By Lemma \ref{finite-real1finite}, after possibly replacing $k$
by a real extension of degree at most $|G_\C/G^0_\C|$, we may and shall assume that $\psi(\xi)=1$,
and hence, $\xi=\varphi(\xi_0)$ for some $\xi_0\in H^1(k, G^0)$.
By assumption, the class $(\xi_0)_\R:=\loc_\R(\xi_0)\in H^1(\RR,G^0)$ maps to  $1\in H^1(\RR,G)$.
It follows that   $(\xi_0)_\R=\delta_\R(g)$ for some $g\in (G\slash G^0)(\RR)$.
By Lemma \ref{finitebdd}, after possibly replacing $k$ by a real extension of degree at most $|G_\C/G^0_\C|$,
we may and shall assume that $g\in (G/G^0)(k)$.
By Serre \cite{SerreCG}, I.5.5, Proposition 39(ii), the class $\xi'_0:=\xi_0*g^{-1}\in H^1(k, G^0)$
maps to $\xi$ under $\varphi$ and becomes 1 over $\RR$.
By assumption, $\xi'_0$ can be  killed by a real extension $K/k$ of degree at most $d_\R(G^0_\C)$.
Clearly, $K$ also kills $\xi$, which completes the proof.
\end{proof}

\begin{subsec}
Next we reduce Theorem \ref{t:cohom-R-bis} to the case when $G_\C$ is reductive.
Let $G_\C$ be a connected linear $\C$-group.
We write $R_u(G_\C)$ for the unipotent radical of $G_\C$\hs,
and we write $G_\C^\red=G_\C/R_u(G_\C)$, which is a connected reductive $\C$-group.
Assume that Theorem \ref{t:cohom-R-bis} is true for $G_\C^\red$.
An argument similar to that of Subsection \ref{ss:pf-gen-0} shows
that then Theorem \ref{t:cohom-R-bis} is also true for $G_\C$ with the same bound
\[ d_\R(G_\C)=d_\R(G^\red_\C).\]
\end{subsec}

\begin{subsec} \emph{Proof of  Theorem \ref{t:cohom-R-bis}.}
We have seen that it suffices to prove the theorem for a connected reductive  $\C$-group.
Let $G_\C$ be such a  $\C$-group, and let $G$ be a $k$-form of $G_\C$.
Let $\xi_G\in H^1(k,G)$ be a cohomology class such that $\loc_\R(\xi_G)=1$.
By Corollary \ref{c:LA}, there exists a \emph{finite} $k$-subgroup $H\subset G$
of order at most $\gamma(G_\C)$ such that $\xi_G$
is the image of some cohomology class $\xi_H\in H^1(k,H)$.
Then by Lemma \ref{l:finite-k}, the class $\xi_H$
can be killed by a finite field extension $M/k$ in $\C$
(not necessarily real, not necessarily normal) of degree at most $|H|\le\gamma(G_\C)$.
Then $\xi_H$ will be killed by the normal closure $L$ of $M$ over $k$
(which is contained in $\C$ and normal over $k$)
of degree at most $\gamma(G_\C)!$ over $k$.
Note that $L$ might not be real.

Set $K=L\cap \R$. Then
\[[K:k]\le [L:k]\le \gamma(G_\C)!\hs.\]
 If $K=L$, then $K/k$ is a real field extension
of degree at most $\gamma(G_\C)!$ \hs killing $\xi_H$ and $\xi_G$\hs, as required.
Otherwise, by Lemma \ref{l:degree-two} we have  $[L:K]=2$.
Set $\Gamma=\Gal(L/K)$. Then $\Gamma=\{1,s\}$, where $s^2=1$.
Consider the restriction $\xi_H^{(K)}=\Res_K(\xi_H)\in H^1(K,H)$.
Since the Galois extension $L/K$ kills $\xi_H^{(K)}$,
by the nonabelian inflation-restriction exact sequence (see Serre \cite{SerreCG}, I.5.8(a)\hs),
the class $\xi_H^{(K)}$ comes from some class $\eta^{(L/K)}\in H^1(L/K, H(L)\hs)$.
Let $a=a^{(L/K)}\colon \Gamma\to H(L)$ be a cocycle representing $\eta^{(L/K)}$.
Then $a_1=1_H$ and $a_s\cdot \hs^s a_s=1_H$.
Let $C$ denote the cyclic subgroup of $H_L$ generated by $a_s$.
Since $^s a_s=a_s^{-1}\in C(L)$, we see that the $L$-subgroup $C$ of $H$
is $\Gamma$-stable, and hence, is defined over $K$.
We regard $a^{(L/K)}$ as an element of $Z^1(L/K,C(L)\hs)$,
and we write $\alpha^{(L/K)}$ for the class of $a^{(L/K)}$ in $H^1(L/K,C(L)\hs)$.
Furthermore, we write $\alpha$ for the image of $\alpha^{(L/K)}$ in $H^1(K,C)$.
Then the image of $\alpha$ in $H^1(K,H)$ is $\xi^{(K)}_H$ and hence,
the image of $\alpha$ in $H^1(K,G)$ is $\Res_K(\xi_G)$.
It follows that the image of $\alpha$ in $H^1(\R,G)$ is 1.
We wish to kill the image $\Res_K(\xi_G)$  of  $\alpha$ in $H^1(K,G)$ by a real extension of bounded degree.

 We consider the homogeneous space  $Y\!:=G_K/C$
and the commutative diagram of pointed sets
\begin{equation*}\label{e:S11}
\xymatrix{
1\ar[r] &Y(K)/G_K(K)\ar[r]^\delta\ar[d]^{\loc_\R} &H^1(K, C)\ar[r]^{i_*}\ar[d]^{\loc_\R} &H^1(K,G_K)\ar[d]^{\loc_\R}\\
1\ar[r] &Y(\R)/G(\R)\ar[r]^{\delta_\R}  &H^1(\R, C)\ar[r]^{i_*}       &H^1(\R,G)\hs,
}
\end{equation*}
where $G_K(K)=G(K)$ and $H^1(K,G_K)=H^1(K,G)$.
By Serre \cite{SerreCG}, I.5.5, Corollary 1 of Proposition 36, the rows of this diagram are exact.
Recall that $\alpha\in H^1(K,C)$.
Consider the localization $\loc_\R(\alpha)\in H^1(\R,C)$.
Then $i_*(\loc_\R(\alpha)\hs)=1\in H^1(\R,G)$,
and hence, $\loc_\R(\alpha)=\delta_\R(o_\R)$ for some orbit $o_\R$ of $G(\R)$ in $Y(\R)$.
Since $G_K$ is connected and $C$ is cyclic, by Theorem \ref{p:LA-cyclic}
the set of $K$-points $Y(K)$ is dense in $Y(\R)$.
Since $o_\R$ is open in $Y(\R)$, we see that $o_\R$ contains a $K$-rational point $y\in Y(K)\cap o_\R$.
Let
\[o=\  G(K)\cdot y\  \in Y(K)/G(K)\]
denote the $G(K)$-orbit of $y$ in $Y(K)$.
Then $\loc_\R(o)=o_\R$ and hence,
\[\loc_\R(\delta(o))=\delta_\R(o_\R)=\loc_\R(\alpha)\in H^1(\R,C).\]

The $K$-group $C$ is cyclic, hence abelian, and therefore,
$H^1(K,C)$ is naturally an abelian group, which we shall write additively.
 Consider $\alpha-\delta(o)\in H^1(K,C)$.
Then
\[\loc_\R(\alpha-\delta(o))=0\in H^1(\R,C).\]
By Corollary \ref{finite-real1finite},  the cohomology class $\alpha-\delta(o)$
can be killed by a real extension $K'$ of $K$ of degree at most $|C|\le|H|\le\gamma(G_\C)$.
This means that over  $K'$,
if we write $\alpha'=\Res_{K'}(\alpha)\in H^1(K',C)$, then  we have $\alpha'=\delta(o')\in H^1(K',C)$
for the $G(K')$-orbit $o'=\,G(K')\cdot y\ \in\, Y(K')/G(K')$.
It follows that
\[i_*(\alpha')=1\in H^1(K',G).\]
Thus we have killed $\Res_K(\xi_G)=i_*(\alpha)$  by a real extension $K'$ of $K$ of degree at most $\gamma(G_\C)$,
and we have killed $\xi_G$ by a real extension $K'$ of $k$ of degree at most $\gamma(G_\C)!\cdot\gamma(G_\C)$.
This completes the proof of Theorem   \ref{t:cohom-R-bis} with the bound
$d_\R(G_\C)=|G_\C/G_\C^0|^2\cdot \gamma(\hs(G^0_\C)^\red)!\cdot\gamma(\hs(G^0_\C)^\red)$,
and thus completes the proofs of Theorems \ref{t:cohom-R} and \ref{t:abs-R} and of Corollary \ref{finite-real}.
\qed
\end{subsec}


\begin{thebibliography}{BDR18}

\bibitem[Bo11]{Borovoi}
M.~Borovoi,
\newblock \textit{Vanishing of algebraic Brauer-Manin obstructions},
J. Ramanujan Math. Soc. \textbf{26} (2011),  333--349.

\bibitem[BDR18]{BDR}
M.~Borovoi, C.~Daw, and J.~Ren,
\newblock  \textit{Conjugation of semisimple subgroups over real number fields of bounded degree},
\newblock \url{arXiv:1802.05894v1 [math.GR].}

\bibitem[CE48]{Chevalley-Eilenberg}
C.~Chevalley and S.~Eilenberg,
\newblock  \textit{Cohomology theory of {L}ie groups and {L}ie algebras},
\newblock Trans. Amer. Math. Soc. \textbf{63} (1948), 85--124.

\bibitem[CT08]{CT}
J.-L.~Colliot-Th\'el\`ene,
 \textit{R\'esolutions flasques des groupes lin\'eaires connexes},
J. Reine Angew. Math. \textbf{618} (2008), 77--133.

\bibitem[DR18]{DawRen17}
C.~Daw and J.~Ren,
\newblock  \textit{Applications of the hyperbolic Ax-Schanuel conjecture},
\newblock Compos. Math. \textbf{154} (2018), no. 9, 1843--1888.


\bibitem[Fr97]{finitegln}
S.~Friedland,
\newblock  \textit{The maximal orders of finite subgroups in {${\rm GL}_n({\bf Q})$}},
\newblock Proc. Amer. Math. Soc. \textbf{125} (1997), 3519--3526.

\bibitem[Hu75]{Humphreys}
J.~E.~Humphreys,
\newblock {\em Linear Algebraic Groups},  Graduate Texts in Mathematics 21,
\newblock Springer-Verlag, New York, 1975.

\bibitem[Kn16]{Knop-MO}
F.~Knop,
(\url{https://mathoverflow.net/users/89948/friedrich-knop}),
 \textit{A finiteness property for semisimple algebraic groups},
URL (version: 2016-04-28): \url{https://mathoverflow.net/q/237585}.

\bibitem[LA15]{LA15}
G.~Lucchini Arteche,
\newblock  \textit{Groupe de Brauer non ramifi\'e alg\'ebrique des espaces homog\`enes},
\newblock Transform. Groups \textbf{20} (2015), 463--493.

\bibitem[Mi17]
{Milne}
J.~S.~Milne,
\newblock {\em Algebraic Groups:
The Theory of Group Schemes of Finite Type over a Field},
Cambridge Studies in Advanced Mathematics  170, Cambridge University Press, Cambridge, 2017.

\bibitem[Ri67]{Richardson}
R.~W.~Richardson, Jr.,
\newblock  \textit{A rigidity theorem for subalgebras of {L}ie and associative algebras},
\newblock Illinois J. Math. \textbf{11} (1967) 92--110.

\bibitem[Sa81]{Sansuc}
J.-J.~Sansuc,
\newblock  \textit{Groupe de {B}rauer et arithm\'etique des groupes alg\'ebriques
lin\'eaires sur un corps de nombres},
\newblock J. Reine Angew. Math. \textbf{327} (1981), 12--80.

\bibitem[Se-LF]{SerreLF}
J.-P.~Serre,
{\em Local fields,}
Graduate Texts in Mathematics 67, Springer-Verlag, New York--Berlin, 1979.

\bibitem[Se-GC]{SerreCG}
J.-P.~Serre,
\newblock {\em Galois cohomology,}
\newblock Springer Monographs in Mathematics, Springer-Verlag, Berlin, 2002.


\end{thebibliography}
\end{document}